\documentclass[12pt]{amsart}
\usepackage{amssymb,latexsym}
\usepackage{amsfonts}
\usepackage{amsmath}
\usepackage[colorlinks,linkcolor=blue,anchorcolor=blue,citecolor=blue]{hyperref}
\usepackage{algorithm}
\usepackage{enumerate}
\usepackage{algpseudocode}
\usepackage{verbatim}
\usepackage{graphicx}

\newcommand\C{{\mathbb C}}

\newcommand\Q{{\mathbb Q}}

\newcommand\Z{{\mathbb Z}}
\newcommand\N{{\mathbb N}}

\newcommand\wH{{\rm H}}
\newcommand\cM{\mathcal{M}}
\newcommand\cD{\mathcal{D}}

\newtheorem{theorem}{Theorem}[section]
\newtheorem{lemma}[theorem]{Lemma}

\newtheorem{corollary}[theorem]{Corollary}

\newtheorem{conjecture}[theorem]{Conjecture}

\theoremstyle{remark}

\numberwithin{equation}{section}

\begin{document}

\title[Multiplicative dependence of translations]{Multiplicative dependence of the translations of algebraic numbers}

\author{Art\= uras Dubickas}
\address{Institute of Mathematics, Faculty of Mathematics and Informatics, Vilnius University, Naugarduko 24, LT-03225 Vilnius, Lithuania}
\email{arturas.dubickas@mif.vu.lt}

\author{Min Sha}
\address{Department of Computing, Macquarie University, Sydney, NSW 2109, Australia}
\email{shamin2010@gmail.com}



\begin{abstract}
In this paper, we first prove that given pairwise distinct algebraic numbers $\alpha_1, \ldots, \alpha_n$, 
the numbers $\alpha_1+t, \ldots, \alpha_n+t$ are multiplicatively independent for all sufficiently large integers $t$. 
Then, for a pair $(a,b)$ of distinct integers, we study how many pairs $(a+t,b+t)$ are multiplicatively dependent when $t$ runs through the set integers $\Z$. 
Assuming the $ABC$ conjecture we show that there exists a constant $C_1$ such that for any pair $(a,b)\in \Z^2$, $a \ne b$, 
 there are at most $C_1$ values of $t \in \Z$ such that $(a+t,b+t)$ are multiplicatively dependent. 
For a pair $(a,b) \in \Z^2$ with difference $b-a=30$ we show that there are $13$
values of $t \in  \Z$ for which the pair $(a+t,b+t)$
is multiplicatively dependent. We further conjecture that $13$ is the largest number of such translations for any such pair $(a,b)$ and prove this for all pairs
$(a,b)$ with difference at most $10^{10}$. 
\end{abstract}

\maketitle

\section{Introduction}  \label{sec:int}

Given $n \ge 1$ non-zero complex numbers $z_1, \ldots, z_n \in \C^*$, we say that they are \textit{multiplicatively dependent} if there exists a non-zero integer vector $(k_1,\dots,k_n) \in \Z^n$ for which
\begin{equation}  \label{eq:MultDep}
z_1^{k_1}\cdots z_n^{k_n}=1.
\end{equation}  
Otherwise (if there is no such non-zero integer vector  $(k_1,\dots,k_n)$), we say that the numbers $z_1,\dots,z_n$ are \textit{multiplicatively independent}.  
Consequently, a vector in $\C^n$ is called \textit{multiplicatively dependent} (resp. \textit{independent}) if its coordinates are all non-zero and  are multiplicatively dependent (resp. independent). To avoid confusion, the vectors with zero coordinates, like $(0,1)$, are not considered to be multiplicatively dependent (although, by convention, $0^0 1^1=1$)
or independent. 

In~\cite{PSSS1}, several asymptotic formulas for the number of multiplicatively dependent vectors of  algebraic numbers of fixed degree  (or lying in a fixed number field) and bounded height have been obtained. In an ongoing project \cite{SSS2}, the authors continue to study multiplicatively dependent vectors from the viewpoint of their density and sparsity. 
By contrast, in this paper aside from the multiplicative dependence and independence of a given set of algebraic numbers we also want to investigate the multiplicative dependence and independence of their translations. More generally, the authors in \cite{OSSZ} study multiplicative dependence of values of rational functions in some special cases.
We remark that a method on deciding the multiplicative independence of complex numbers in a finitely generated field  has been 
proposed by Richardson \cite{Richardson}. 

In Section~\ref{sec:ind} (Theorem~\ref{thm:mult ind1}), we prove a result which implies that given pairwise distinct algebraic numbers $\alpha_1, \ldots, \alpha_n, n\ge 2$, for each sufficiently large integer $t$, 
the algebraic numbers $\alpha_1+t, \ldots, \alpha_n+t$ are multiplicatively independent. 
This is in fact a special case of \cite[Theorem 1']{BMZ}. A weaker version of this statement given in \cite[Lemma 2.1]{Dubickas} was used 
in \cite{Dubickas} and so it is an additional motivation for 
Theorem~\ref{thm:mult ind1}.
In particular, by Theorem~\ref{thm:mult ind1}, for an integer vector $(a_1,\ldots,a_n)$ whose coordinates are pairwise distinct, there are only finitely many integers $t$ for which the numbers $a_1+t,\ldots, a_n+t$ are multiplicatively dependent. 
So, a natural question is to estimate the number of such integers $t$ corresponding to a given integer vector. 
In this paper, we investigate in detail the case of dimension $n=2$ by presenting some explicit formulas, upper bounds and several conjectures. See Theorems~\ref{thm:size2}, \ref{thm:size23}, \ref{thm:upper} and \ref{thm:uniform}. 
For example, we conjecture that for any pair of distinct integers $(a,b) \in \Z^2$, the number of such integer translations $t$ is at most 13, 
which is in fact related to two special forms of Pillai's equation. 
The pair $(a,b)=(1,31)$ is an example which has exactly $13$ integer translations leading to multiplicatively dependent vectors (see Section~\ref{sec:set}).

\section{Preliminaries}  \label{sec:pre}

For the convenience of the reader, we recall some basic concepts and results in this section, which are used later on. 

For any algebraic number $\alpha$ of degree $\deg \alpha =m\ge 1$, let
$$
f(x)=a_mx^m+\cdots+a_1x+a_0
$$
be the minimal polynomial of $\alpha$ over the integers $\Z$, where $a_m>0$. Suppose that $f$ is factored as
$$
f(x)=a_m(x-\alpha_1)\cdots (x-\alpha_m)
$$
over the complex numbers $\mathbb{C}$. 
The \textit{height} of $\alpha$, also known as the \textit{absolute Weil height} of $\alpha$ and denoted by $\wH(\alpha)$, is defined by
\begin{equation*}
\wH(\alpha)=\big(a_m\prod_{i=1}^{m}\max\{1,|\alpha_i|\}\big)^{1/m}.
\end{equation*}
Besides, we define the \textit{house} of  $\alpha$ to be the maximum of the modulus of its conjugates: 
$$
\overline{|\alpha|} = \max \{ |\alpha_1|, \ldots, |\alpha_m| \}; 
$$
see \cite[Section 3.4]{Waldschmidt2000}. Clearly, if $|a_0/a_m|\ge 1$ we have 
$$
\wH(\alpha) \le a_m^{1/m} \overline{|\alpha|}.
$$
In particular, for any algebraic integer $\alpha \ne 0$ we have $\wH(\alpha) \le \overline{|\alpha|}$.

The next result shows that if algebraic numbers $\alpha_1,\ldots,\alpha_n$ are multiplicatively dependent, then one can find a relation as in \eqref{eq:MultDep}, where the exponents $k_i$, $i=1,\dots,n$, are  not too large; 
see for example~\cite[Theorem 3]{Loxton} or \cite[Theorem~1]{Poorten}.

\begin{lemma} 
\label{lem:exponent}
Let $n\geq 2$, and let $\alpha_1,\ldots,\alpha_n$ be multiplicatively dependent non-zero algebraic numbers of height at most $H \ge 2$ 
and contained in a number field $K$ of degree $D$ over the rational numbers $\Q$. Then, there are $k_1,\ldots,k_n \in \Z$, not all zero, and  a positive number $c_1$ which depends only on $n$,  
such that
\begin{equation}\label{mkoi}
\alpha^{k_1}_1\cdots\alpha^{k_n}_n=1
\end{equation}
and
\begin{equation}  
\label{eq:exp1}
\max_{1\leq i\leq n}|k_i| \le c_1D^n(\log (D+1))^{3(n-1)}(\log H)^{n-1}.
\end{equation}
Furthermore, if $K$ is totally real, then there are integers $k_1,\ldots, k_n$, not all zero, as in \eqref{mkoi} and  a positive number $c_2$ which depends only on $n$ such that 
\begin{equation}  
\label{eq:exp2}
\max_{1\leq i\leq n}|k_i| \le c_2(\log H)^{n-1}.
\end{equation}
\end{lemma} 

\begin{proof}
Let $w(K)$ be the number of roots of unity in $K$. 
Note that for Euler's totient function $\varphi$ we have $\varphi(m)\gg m/\log\log m$ for any $m\ge 3$. 
Since $\varphi(w(K)) \le D$, we obtain $w(K)\ll D\log\log (3D)$. 
Then, using \cite[Theorem 3 (A)]{Loxton} we can get \eqref{eq:exp1}. 
In the same fashion, \eqref{eq:exp2} follows directly from \cite[Theorem 3 (B)]{Loxton}. 
\end{proof}

The following statement is Mih{\u a}ilescu's theorem (previously known as Catalan's conjecture) \cite{Mih}, 
which roughly says that $(2^3,3^2)$ is the only case of two consecutive powers of natural integers.  

\begin{lemma}[\cite{Mih}] \label{lem:Catalan}
The equation 
$$
b^y - a^x = 1
$$ 
with unknowns
$b\ge 1, y \ge 2,  a \ge 1, x \ge2$ 
has only one integer solution $(a,b,x,y)=(2,3,3,2)$. 
\end{lemma}  

We also need the following classical result due to Siegel \cite{Siegel}. 

\begin{lemma}[\cite{Siegel}]   \label{lem:Siegel}
Let $f(x)$ be a polynomial in $\Z[x]$. If $f$ has at least three simple roots, then the equation $y^2=f(x)$ has only 
finitely many integer solutions $(x,y)$. 
\end{lemma}

\section{Multiplicative independence}  \label{sec:ind}

In the following theorem, we confirm the multiplicative independence among the translations of algebraic numbers. 
Actually, we can do more than it was claimed at the beginning. 

\begin{theorem} \label{thm:mult ind1}
Let $\alpha_1, \ldots, \alpha_n$ be pairwise distinct algebraic numbers, and let
$d=[\Q(\alpha_1,\dots,\alpha_n):\Q]$. Then, there is a positive constant $C=C(n,\alpha_1,\ldots, \alpha_n)$ 
such that for any algebraic integer $t$ of degree at most $\overline{|t|}^{1/(nd+1)}$ and with $\overline{|t|}\ge C$, 
the following $n$ algebraic numbers $\alpha_1+t, \ldots, \alpha_n+t$ are multiplicatively independent. 
\end{theorem}

We remark that the exponent $1/(nd+1)$ for $\overline{|t|}$ here is not optimal and is chosen for the sake of simplicity. 

\begin{proof}
The result is trivial for $n=1$. Assume that $n \ge 2$. 
Without loss of generality, we can further assume that 
\begin{equation}  \label{eq:hou}
|t| = \overline{|t|}.
\end{equation}
Indeed, if $|t| \ne \overline{|t|}$, then there is a Galois isomorphism $\sigma$ of the Galois closure of $\Q(\alpha_1,\dots,\alpha_n,t)$ over $\Q$ such that $|\sigma(t)| = \overline{|t|}$.
Then, it suffices to verify the multiplicative independence of the algebraic numbers $\sigma(\alpha_1)+\sigma(t), \ldots, \sigma(\alpha_n)+\sigma(t)$. 

Take $|t|$ large enough. Then, we can assume that 
$\alpha_i+t \ne 0$ and, moreover, 
$$
\big| |1+\alpha_i/t| - 1 \big| < \varepsilon, \quad i=1,2,\ldots, n,
$$
for a sufficiently small $\varepsilon>0$. 
For a complex number $z$, let $\arg(z) \in (-\pi, \pi]$ be the principal argument of $z$. 
Note that for $\varepsilon \le 1/2$ and each $i=1,2,\ldots, n$, we have 
$$
|\sin(\arg(1 + \alpha_i / t))| =\frac{|\sin(\arg (\alpha_i/t))|\cdot |\alpha_i/t|}
{ |1+\alpha_i/t|} \le 2 |\alpha_i| / |t|. 
$$
Thus, using the fact that $|x| \le 2|\sin x| $ for any $x
\in [-\pi/2,\pi/2]$, 
we can further assume that the principal arguments satisfy
\begin{equation}\label{vienas}
|\arg(1 + \alpha_i / t)| \le 4|\alpha_i|/|t|, \quad i=1,2,\ldots, n.
\end{equation}
Besides, by the basic properties of the Weil height (see, e. g., 
\cite{Waldschmidt2000}) and \eqref{eq:hou}, we have 
\begin{equation}  \label{eq:Halt}
\wH(\alpha_i+t) \le 2\wH(t)\wH(\alpha_i) \le 2|t|\wH(\alpha_i),  \quad i=1,2,\ldots, n.
\end{equation}
Here, $\wH(t) \le |t|$, since $t$ is an algebraic integer and $|t| = \overline{|t|}$, by \eqref{eq:hou}. 

For a contradiction, assume that $\alpha_1+t, \ldots, \alpha_n+t$ are multiplicatively dependent,
that is, there is a non-zero vector $(k_1,\ldots,k_n) \in \Z^n$ such that 
\begin{equation} \label{eq:mult t1}
(\alpha_1+t)^{k_1} \cdots (\alpha_n+t)^{k_n} = 1. 
\end{equation}
Set $$D=[\Q(\alpha_1,\dots,\alpha_n,t):\Q].$$
Then, by the degree assumption on $t$, we find that
$$D \leq [\Q(t):\Q]d \leq d|t|^{1/(nd+1)}.$$
By Lemma~\ref{lem:exponent} (see \eqref{eq:exp1}) and \eqref{eq:Halt}, we can further assume that the nonzero integers in \eqref{eq:mult t1} can be chosen such that
\begin{equation}\label{antrokas}
\max_{1\leq i\leq n}|k_i| \le c_3 |t|^{n/(nd+1)} (\log |t|)^{4(n-1)},
\end{equation}
where $c_3$ depends only on $n,\alpha_1,\dots,\alpha_n$. 
(Note that $d$ also depends on $\alpha_1,\dots,\alpha_n$.)

Observe first that if in \eqref{eq:mult t1} we have $S=\sum_{i=1}^{n} k_i \ne 0$, then, since each $|\alpha_i + t|$ is close to $|t|$, the absolute value of the left-hand side of \eqref{eq:mult t1} 
is either very large (if $S>0$) or very small 
(if $S<0$)
provided that $|t|$ is large enough, which contradicts with \eqref{eq:mult t1}.  
Indeed, by \eqref{eq:mult t1}, we obtain 
$$
|t|^S = \prod_{i=1}^{n} |1+\alpha_i/t|^{-k_i}. 
$$
Suppose that $S \ne 0$. Replacing $(k_1,\ldots,k_n)$ by $(-k_1,\ldots,-k_n)$ if necessary, we can assume that $S>0$, and hence $S\ge 1$. 
Then, using $|1+\alpha_i/t|^{-k_i} \le |1+|\alpha_i|/|t||^{2|k_i|}$ for $|t|$ large
enough and $|t| \le |t|^S$,  we deduce that 
$$
|t| \le \prod_{i=1}^{n} (1+ |\alpha_i|/|t|)^{2|k_i|} \le \exp \Big(\frac{2}{|t|}\sum_{i=1}^{n}|k_i||\alpha_i|\Big). 
$$
By taking logarithms of both sides and using
\eqref{antrokas}, we get the inequality
$$|t|\log |t| \le  c_4 |t|^{n/(nd+1)} (\log |t|)^{4(n-1)}$$ 
for some constant $c_4$ depending only on $n,\alpha_1,\ldots,\alpha_n$. 
However, this inequality cannot hold for $|t|$ large enough, because $n/(nd+1)<1$.
Thus, we must have $S= 0$. 

Now, by \eqref{eq:mult t1} combined with $\sum_{i=1}^{n} k_i = 0$, it follows that
\begin{equation} \label{eq:mult t2}
(1+\alpha_1/t)^{k_1} \cdots (1+\alpha_n/t)^{k_n} = 1. 
\end{equation}
With our assumptions, by \eqref{vienas}, we further deduce that
$$
\sum_{i=1}^{n} |k_i \arg(1+\alpha_i/t)| \le \sum_{i=1}^{n} \frac{4|k_i \alpha_i|}{|t|},
$$
which, by \eqref{antrokas}, is clearly less than  $\pi$ when $|t|$ is large enough. 
So, by taking logarithms of both sides of \eqref{eq:mult t2}, we obtain 
$$
\sum_{i=1}^{n} k_i \log (1+ \alpha_i/t) = 0, 
$$
where ``$\log$" means the principal branch of the complex logarithm. 
Then, using the Taylor expansion we deduce that
\begin{equation} \label{eq:Taylor}
\frac{1}{t} \sum_{i=1}^{n} k_i \alpha_i  - \frac{1}{2t^2} \sum_{i=1}^{n} k_i \alpha_i^2  + \frac{1}{3t^3} \sum_{i=1}^{n} k_i \alpha_i^3 - \cdots = 0.
\end{equation} 

Multiplying both sides of \eqref{eq:Taylor} by $t$ and using the bound \eqref{antrokas}, we get 
\begin{equation} \label{eq:sum1}
\big|\sum_{i=1}^{n} k_i \alpha_i \big| \le c_5 |t|^{(n-nd-1)/(nd+1)} (\log |t|)^{4(n-1)}, 
\end{equation}
where $c_5$ is a constant depending only on $n$ and $\alpha_1,\ldots, \alpha_n$. 

Assume that $\sum_{i=1}^{n} k_i \alpha_i \ne 0$. Then, by Liouville's inequality (see \cite[Proposition 3.14]{Waldschmidt2000}) and the upper bound \eqref{antrokas}, one can easily get that
\begin{equation} \label{eq:sum2}
\big| \sum_{i=1}^{n} k_i \alpha_i \big| \ge c_6(|t|^{n/(nd+1)} (\log |t|)^{4(n-1)})^{1-d}, 
\end{equation}
where $c_6$ is a constant depending only on $n$ and $\alpha_1, \ldots, \alpha_n$. 
Clearly, in view of $nd-n+1>n(d-1)$ the two estimates \eqref{eq:sum1} and \eqref{eq:sum2} lead to a contradiction provided that $|t|$ is large enough. 
Hence, we must have 
$$
\sum_{i=1}^{n} k_i \alpha_i = 0. 
$$

Applying the same argument to \eqref{eq:Taylor}, step by step, we obtain 
\begin{align*}
  \sum_{i=1}^{n} k_i \alpha_i ^2= 0, 
\quad  \sum_{i=1}^{n} k_i \alpha_i^3 = 0,  
\quad \ldots,
\quad \sum_{i=1}^{n} k_i \alpha_i^n = 0. 
\end{align*}
This is a system of $n$ linear equations with unknowns $k_1,\ldots, k_n$.  Notice that its coefficient matrix is the Vandermonde matrix 
with non-zero determinant, since $\alpha_i \ne \alpha_j$ for $1\le i \ne j \le n$. 
So, we must have 
$$
k_1 = \ldots = k_n =0, 
$$ 
which contradicts to the assumption that $(k_1,\ldots, k_n)$ is a 
non-zero vector. 
This completes the proof of the theorem. 
\end{proof}

Following the same arguments as in the proof of Theorem~\ref{thm:mult ind1} and using the inequality \eqref{eq:exp2} of Lemma~\ref{lem:exponent} (instead of \eqref{eq:exp1}) which yields $$\max_{1\leq i\leq n}|k_i| \le c_7(\log |t|)^{n-1}$$ instead of \eqref{antrokas},
we obtain the following:

\begin{theorem} \label{thm:mult ind2}
Given $n \geq 2$ pairwise distinct totally real algebraic numbers $\alpha_1, \ldots, \alpha_n$, there is a positive constant $C=C(n,\alpha_1,\ldots, \alpha_n)$ 
such that for any totally real algebraic integer $t$ with $\overline{|t|}\ge C$, 
the following $n$ algebraic numbers $\alpha_1+t, \ldots, \alpha_n+t$ are multiplicatively independent. 
\end{theorem}

Theorem~\ref{thm:mult ind1} implies the following corollary.
(It also follows from \cite[Theorem 1']{BMZ}, 
by considering the line parameterized by $x-\alpha_1,\ldots,x-\alpha_n$ as $x$ varies.)  

\begin{corollary}
Given a positive integer $m$ and $n \geq 2$ pairwise distinct algebraic numbers $\alpha_1, \ldots, \alpha_n$, there is a positive constant $C=C(m,n,\alpha_1,\ldots, \alpha_n)$ 
such that for any algebraic integer $t$ of degree at most 
$m$ and with $\overline{|t|}\ge C$, 
the following $n$ algebraic numbers $\alpha_1+t, \ldots, \alpha_n+t$ are multiplicatively independent. 
\end{corollary}

In particular, we have: 

\begin{corollary}\label{darkart}
Given $n$ pairwise distinct algebraic num\-bers $\alpha_1, \ldots, \alpha_n$, there are only finitely many integers $t \in \Z$ for which the translated numbers
$\alpha_1+t, \ldots, \alpha_n+t$ are multiplicatively dependent. 
\end{corollary}

On the other hand, for a fixed integer $t\in \Z$, there are infinitely many vectors $(\alpha_1,\ldots,\alpha_n) \in \Z^n$ such that 
$(\alpha_1+t,\ldots,\alpha_n+t)$ is multiplicatively independent. For example, we can choose $\alpha_i=p_i-t$ for each $i$, 
where $p_1,\ldots,p_n$ are pairwise distinct rational primes.

\section{Sets of multiplicatively dependent vectors}  \label{sec:set}

\subsection{General setting}

In this section, we focus our attention on vectors in $\Z^2$ which are multiplicatively dependent. 
This turns out to be related to Pillai's equation, which is a quite typical kind of Diophantine equation and has been extensively studied; see, for example, \cite{Bennett,Bugeaud,Scott}. 

Starting from an integer vector $(a_1,\ldots, a_n) \in \Z^n$, we can get a set of multiplicatively dependent vectors 
in $\Z^n$ by adding $t \in \Z$ to each coordinate of the given vector. 
Corollary~\ref{darkart} implies that the set of such $t \in \Z$ is finite when the coordinates of the given vector are pairwise distinct, namely, $a_i \ne a_j$ for $i \ne j$.
Now, a natural question is to estimate the size of the set of possible $t \in  \Z$ for which the vector $(a_1+t,\dots,a_n+t)$  is multiplicatively dependent (and thus contains no zero coordinates by definition).  
In this paper, we only consider the simplest case $n=2$. 

Given a vector $(a,b) \in \Z^2$ with $a \ne b$, note that either $(1,b-a+1)$ or $(-1,b-a-1)$ is multiplicatively dependent  obtained from $(a,b)$ by translation as above, because $b-a+1$ and $b-a-1$ cannot be zero at the same time. 
So, the set of all possible $t \in \Z$ only depends on the difference $b-a$, which is also called the \textit{difference} of the set. 
For an integer $d\in \Z$, we denote by $\cM(d)$ the set of multiplicatively dependent vectors in $(a,b) \in \Z^2$, $ab \ne 0$, with difference $d=b-a$. 
Corollary~\ref{darkart} implies that each set $\cM(d), d\ne 0$, is a finite set.  
Let us put
$$
M(d) = | \cM(d) |,  \qquad d\in \Z, 
$$
where $|\cM(d)|$ is the cardinality of the set $\cM(d)$. 
One interesting direction is to study the size of $M(d)$, and especially whether the following maximum  
$$
\max_{d\ne 0} M(d) 
$$
is finite. 
(Clearly, the set $\cM(0)$ is infinite, because it consists of all pairs $(a,a)\in \Z^2, a\ne 0$.)  

Note that for any multiplicatively dependent vector $(a,b) \in \Z^2$, we certainly have $(a,b) \in \cM(b-a)$.  
So, the sets $\cM(d), d\in \Z$, form a disjoin union of all the multiplicatively dependent vectors in $\Z^2$. 
Since there is a one-to-one correspondence between the vectors in $\cM(d)$ and those in $\cM(-d)$ by 
the permutation of coordinates, we have 
$$
M(d) = M(-d)
$$
 for any $d \ne 0$. So, in the sequel we will always assume that $d \in \N$. 

Before going further, let us emphasize the following useful fact about multiplicatively dependent vectors in $\Z^2$. 
That is, if $(a,b)\in \Z^2$, $a \ne b$, is multiplicatively dependent, then there exists a positive integer $g$ and two non-negative integers $x,y$ such that $(a,b)=(\pm g^x, \pm g^y)$.

\subsection{Some explicit formulas} \label{sec:explicit}

We essentially relate $M(d)$ to counting integer solutions of two simple Pillai's equations in the lemma below. 

Throughout, for any given integer $d \geq 1$ we say that an integer solution $(g,x,y)$ of the equation 
\begin{equation} \label{eq:Pillai1} 
g^y+g^x = d, \qquad g\ge 2 \quad \text{and} \quad y>x \ge 1 
\end{equation} 
is \textit{primitive} 
if $g$ is not a perfect power. Let $N^{+}(d)$ be the number of primitive integer solutions of \eqref{eq:Pillai1}. Similarly, for any given integer $d \geq 1$ we say that an integer solution $(g,x,y)$ of the equation 
\begin{equation} \label{eq:Pillai2}
g^y-g^x = d, \qquad g\ge 2 \quad \text{and} \quad y>x \ge 1
\end{equation} 
is \textit{primitive} 
if $g$ is not a perfect power. Let $N^{-}(d)$ be the number of primitive integer solutions of \eqref{eq:Pillai2}.

\begin{lemma} \label{lem:size}
For any integer $d \ge 3$, we have
\begin{equation}\label{countt}
M(d) = 2N^{+}(d)+2N^{-}(d) + 4+\delta(d),  
\end{equation}
where $\delta(d)=1$ if $d$ is even, and  $\delta(d)=0$ if $d$ is odd. 
\end{lemma} 

\begin{proof} 
Let 
$$
S_0 = \{(-d-1,-1), (-d+1,1), (-1,d-1),(1,d+1)\}, 
$$
\begin{align*}
S_1 = \{(-g^x,g^y),(-g^y,g^x):  \textrm{$(g,x,y)$ is a primitive solution of \eqref{eq:Pillai1}} \}
\end{align*}
and 
$$
S_2 = \{(g^x,g^y),(-g^y,-g^x): \textrm{$(g,x,y)$ is a primitive solution of \eqref{eq:Pillai2}} \}.
$$
We claim that 
\begin{equation}\label{countt1}
 \cM(d) = S_0 \cup S_1 \cup S_2
 \end{equation}
 if $d$ is odd, and
\begin{equation}\label{countt2}
\cM(d) = \{(-d/2,d/2) \} \cup S_0 \cup S_1 \cup S_2 
 \end{equation}
if $d$ is even. 

Evidently, $S_0 \subseteq \cM(d)$. Also, $(-d/2,d/2) \in \cM(d)$ if $d$ is even. 
Let us count the vectors $(a,b) \in \cM(d) \setminus \big(S_0 \cup \{(-d/2,d/2)\}\big)$ with $ab<0$. Then,
$a<0<b$, so that such vectors $(a,b)$ have a form of $(-g^x,g^y)$ or $(-g^y,g^x)$ for some positive integer $g\ge 2$ and two non-negative integers $x \le y$.  
If $d$ is even, then $(-d/2,d/2) \in \cM(d)$, which corresponds to the case $g^x=g^y=d/2$, so this solution is not in $S_1 \cup S_2$, and since $d \ge 3$, we have $(-d/2,d/2) \notin S_0$. 
In case $x=0$, that is, $g^x=1$, we obtain two vectors 
$(-d+1, 1), (-1,d-1) \in \cM(d)$, which are already in $S_0$.
Besides, if an integer vector $(g,x,y)$ with $g=a^r\ge 2$ and $y>x \ge 1$ satisfies $g^y+g^x=d$, where $a$ and $r$ are positive integers, then $(g,x,y)$ and $(a,rx,ry)$ are different integer solutions 
of \eqref{eq:Pillai1}, but they produce the same vectors in $\cM(d)$: $(-g^x,g^y)$ and $(-g^y,g^x)$. 
Thus, the sets $\{(-d/2,d/2)\} \cup S_0$ and $S_1$ are disjoint and, by the definition of $S_1$, we have 
$|S_1|=2N^{+}(d)$. 
 
It remains to count the vectors $(a,b)\in \cM(d)$ with $ab>0$. Clearly,
they have the form $(g^x,g^y)$ or $(-g^y,-g^x)$ for some positive integer $g\ge 2$ and two non-negative integers $x,y$ with $y>x\ge 0$. 
If $x=0$, i. e., $g^x=1$, we get two vectors 
$(-d-1, -1), (1,d+1) \in \cM(d)$ which belong to $S_0$.
Now, by the same argument as the above, we see that 
the sets $\{(-d/2,d/2)\} \cup S_0$ and $S_2$ are disjoint and 
$|S_2|=2N^{-}(d)$. 

Finally, since the sets $S_1$ and $S_2$ are disjoint by their definitions 
(and each of them is disjoint from the set $\{(-d/2,d/2)\} \cup S_0$),
 we deduce \eqref{countt}, in view of 
\eqref{countt1}, \eqref{countt2}, 
$|S_0|=4$, $|S_1|=2N^{+}(d)$
and $|S_2|=2N^{-}(d)$.  
\end{proof}

Lemma \ref{lem:size} transfers our problem to estimates for the quantities $N^{+}(d)$ and $N^{-}(d)$.  
Next, using the formulas \eqref{countt1} and \eqref{countt2} we give the explicit constructions for $\cM(d)$ as well as the explicit values for their sizes $M(d)$ in some special cases.  

\begin{theorem} \label{thm:size2}
We have
\begin{itemize}
\item[{\rm (i)}] $M(1)=2, M(2)=5$, and $M(2^r) = 7$ for any positive integer $r \ge 2$;

\item[{\rm (ii)}] $M(d)= 4$ for any odd integer $d\ge 3$.   
\end{itemize}
\end{theorem}

\begin{proof}
It is straightforward to check that   
$$
\cM(1) = \{ (-2,-1), (1,2)\}, 
$$
and 
$$
\cM(2) = \{(-4,-2), (-3,-1), (-1,1), (1,3), (2,4) \}. 
$$ 

Now, we consider the set $\cM(2^r)$, where $r \ge 2$. 
We first look at the equation \eqref{eq:Pillai1} with $d=2^r$. 
Notice that $g^x(g^{y-x}+1) = 2^r$. 
Since $x \ge 1$ and $\gcd(g^x,g^{y-x}+1)=1$, the left-hand side $g^x(g^{y-x}+1)$ has at least two distinct prime factors. 
So, there is no integer solution of the equation \eqref{eq:Pillai1}.
Consequently, $N^{+}(2^r)=0$. 

Next, let us consider the equation \eqref{eq:Pillai2} with $d=2^r$. 
This time, in view of $g^x(g^{y-x}-1) = 2^r$ and $x \ge 1$, we must have
$g^x=2^r$ and $g^{y-x}=2$. Hence,
 $(g,x,y)=(2,r,r+1)$ is the only primitive integer solution of \eqref{eq:Pillai2}. 
It follows that $N^{-}(2^r)=1$, which gives two vectors 
$$
(-2^{r+1},-2^r), (2^r,2^{r+1}) \in \cM(2^r). 
$$ 

So, by Lemma~\ref{lem:size}, it follows that $M(2^r)=2 \cdot 0 + 2\cdot 1+4+1=7$ for $r \ge 2$, as claimed.
 This completes the proof of (i).

Now, let $d\ge 3$ be an odd integer. 
Considering the equation \eqref{eq:Pillai1},  we first note that, since $x \ge 1$, it is impossible to have $g^y+g^x = d$ for $d$ odd, because $g^y+g^x$ is even.  
Similarly, there is also no integer solution of the equation \eqref{eq:Pillai2} for $d$ odd. 
Using $N^{+}(d)=N^{-}(d)=0$, by Lemma~\ref{lem:size}, we obtain $M(d)=4$, as claimed in (ii), and in fact 
$$
\cM(d) = \{(-d-1,-1), (-d+1,1), (-1,d-1),(1,d+1)\}
$$
for each odd $d \geq 3$.
\end{proof}

To handle the case 
when $d$ is the product of a power of 2 and a power of an odd prime, i. e., 
 $d=2^r p^s$, where $p \ge 3$ is a prime and $r, s\ge 1$,
we shall
 use Mih{\u a}ilescu's theorem, that is, 
 Lemma \ref{lem:Catalan}.
Recall that a prime number $p$ is said to be a \textit{Fermat prime} if $p=2^m+1$ for some positive integer $m$, 
and consequently $m$ must be a power of $2$.  
So far, the only known Fermat primes are $3, 5, 17, 257, 65537$. 
Also, recall that a prime number $p$ is called a \textit{Mersenne prime} if $p=2^m-1$ for some positive integer $m$, 
and in fact $m$ must be also a prime.

\begin{theorem} \label{thm:size23}
 Let $r$ and $s$ be two positive integers. For $1\le r \le 3$ we have 
\begin{equation*}
M(2^r3^s) = 
\begin{cases}
11 & \textrm{if } s=1, \\
9 & \textrm{if } s=2, \\
7 & \textrm{if } s \ge 3; 
\end{cases} 
\end{equation*} 
for $r\ge 4$, we have 
\begin{equation*}
M(2^r3^s) = 
\begin{cases}
9 & \textrm{if } s=1, \\
7 & \textrm{if } s=2, \\
5 & \textrm{if } s \ge 3. 
\end{cases} 
\end{equation*} 

Let $p \ge 5$ be a prime, and let $r, s$ be two positive integers. Then, 
\begin{equation*}
M(2^rp^s)= 
\begin{cases}
9 & \textrm{if $s=1$, and either $p=2^r+1$ or $p=2^r-1$}, \\
7 & \textrm{if $s\ge 2$, and either $p=2^r+1$ or $p=2^r-1$}, \\
7 & \textrm{if $s=1$, and either $p$ is a Fermat prime satisfying } \\
    &  \textrm{$p \ne 2^r+1$,  or 
$p$ is a Mersenne prime satisfying}\\
&  \textrm{$p \ne 2^r -1$}, \\
5 & \textrm{otherwise}.
\end{cases} 
\end{equation*} 
\end{theorem}

\begin{proof}

By Lemma~\ref{lem:size}, it suffices to count primitive integer solutions of the equations \eqref{eq:Pillai1} and \eqref{eq:Pillai2}. 

Consider the equation \eqref{eq:Pillai1} with $d=2^r p^s$, where $p \ge 3$ is a prime. 
Since $y>x \ge 1$, from $g^y+g^x = g^x(g^{y-x}+1) = 2^r p^s$, 
we must have 
\begin{equation}  \label{eq:23+}
\textrm{either} \quad 
\begin{cases}
g^x = 2^r \\
g^{y-x} + 1 = p^s 
\end{cases} 
\quad \textrm{or} \quad 
\begin{cases}
g^x = p^s \\
g^{y-x} +1 = 2^r. 
\end{cases}
\end{equation}

In the first case, since $g$ is not a perfect power, we must have
$g=2$ and $x=r$. 
The second equation $g^{y-x} + 1 = p^s$ becomes 
\begin{equation} \label{eq:23_1}
2^{y-r} + 1 = p^s.
\end{equation}
By Lemma \ref{lem:Catalan}, in \eqref{eq:23_1} we cannot have $s \ge 3$. Suppose
that in \eqref{eq:23_1} we have $s=2$.
Then, by Lemma \ref{lem:Catalan}, $p=3$ and $y=r+3$.
This gives the unique primitive solution $(g,x,y)=(2,r,r+3)$
of \eqref{eq:Pillai1}. 
If in \eqref{eq:23_1} we have $s=1$ , then there is a unique primitive solution
  of \eqref{eq:Pillai1} if and only if
$p$ is a Fermat prime. (Otherwise, \eqref{eq:Pillai1} has no primitive solutions.) Consequently, the contribution of the ``first case" into
the quantity
$N^{+}(2^r p^s)$ is one if $(p,s)=(3,2)$ or if $p$ is a Fermat prime and $s=1$, and zero otherwise. 

In the second case of \eqref{eq:23+}, we must have $g=p$ and $x=s$. 
The second equation $g^{y-x} + 1 = 2^r$ becomes 
\begin{equation} \label{eq:23_2}
p^{y-s} + 1 = 2^r.
\end{equation} 
Clearly, $r \ge 2$. 
Note that we cannot have $y-s\ge 2$ in \eqref{eq:23_2}, by Lemma \ref{lem:Catalan}.  
Hence, $y=s+1$. This yields $p=2^r-1$.  Hence, 
 the contribution of the ``second case" of \eqref{eq:23_2} into the quantity
$N^{+}(2^r p^s)$
is one if and only if $p=2^r-1$, where $r \ge 2$, and zero otherwise. 
Combining both these contributions we deduce that
\begin{equation}\label{durn1}
N^{+}(2^r p^s) = 
\begin{cases}
2 & \textrm{if } p=3, r=2, s \in \{1,2\},\\
1 & \textrm{if } p=3, r \ne 2, s \in \{1,2\}, \\
1 & \textrm{if } p=3, r=2, s \ge 3, \\
1 & \textrm{if } p \ge 5 \>\> \textrm{is a Fermat prime and } s=1, \\
1 & \textrm{if } p =2^r-1 \>\> \textrm{and } r \ge 3, \\
0 & \textrm{otherwise}.
\end{cases} 
\end{equation}

Now, let us investigate the equation \eqref{eq:Pillai2} with $d=2^r p^s$.
Since $y>x \ge 1$, by $g^y-g^x = g^x(g^{y-x}-1) = 2^r p^s$, 
we must have 
\begin{equation} \label{eq:23-}
\textrm{either} \quad 
\begin{cases}
g^x = 2^r \\
g^{y-x} - 1 = p^s 
\end{cases} 
\quad \textrm{or} \quad 
\begin{cases}
g^x = p^s \\
g^{y-x} - 1 = 2^r. 
\end{cases}
\end{equation}
In the first case of \eqref{eq:23-}, we obtain $(g,x)=(2,r)$, and 
the second equation $g^{y-x} - 1 = p^s$ becomes 
\begin{equation} \label{eq:23_3}
2^{y-r} - 1 = p^s.
\end{equation} 
Clearly, we must have $y-r\ge 2$. 
By Lemma \ref{lem:Catalan}, the equality in \eqref{eq:23_3} can not hold for $s \ge 2$. 
For $s=1$ there is a unique integer solution of \eqref{eq:23_3}
if and only if $p$ is a Mersenne prime. 

In the second case of \eqref{eq:23-}, we obtain $(g,x)=(p,s)$.  
The second equation $g^{y-x} - 1 = 2^r$ becomes 
$$
p^{y-s} - 1 = 2^r.
$$
For $r=1$ we obtain $p=3$ and $y=s+1$. For $r=3$, we must have
$p=3$ and $y=s+2$. Then, for $r \in \N \setminus \{1,3\}$, by Lemma~\ref{lem:Catalan}, we must
have $y=s+1$ and so $p$ is a Fermat prime of the form $p=2^r+1$. 
Therefore, as above, combining both contributions into $N^{-}(2^rp^s)$ we derive that
\begin{equation}\label{durn2}
N^{-}(2^r p^s) = 
\begin{cases}
2 & \textrm{if } p=3, s=1, r \in \{1,3\},\\
1 & \textrm{if } p=3, s=1, r \notin \{1,3\},\\
1 & \textrm{if } p=3, s \ge 2, r \in \{1,3\},\\
1 & \textrm{if } p \ge 7 \>\> \textrm{is a Mersenne prime and } s=1, \\
1 & \textrm{if } p =2^r+1 \>\> \textrm{and } r \ge 2, \\
0 & \textrm{otherwise}.
\end{cases} 
\end{equation}

Finally, applying Lemma~\ref{lem:size} and combining  
\eqref{durn1} with \eqref{durn2} first for $p=3$ and then for $p \geq 5$, we conclude the proof. 
\end{proof}

Obviously, given an explicit value of $d$, following the arguments in the proof of Theorem~\ref{thm:size23} 
we can compute the exact value of $M(d)$. 
However, the argument can be quite complicated when $d$ has many distinct prime factors. 
At the end of the paper we will present an algorithm which allows to calculate $M(d)$ for any given even integer $d \in \N$.

\subsection{Unconditional upper bound}

Note that in the above we have obtained the exact value of $M(d)$ when $d$ is either odd or has at most two distinct 
prime factors. 
Now, we present an unconditional upper bound for $M(d)$ when $d$ is even and has at least three distinct prime factors. 

\begin{theorem} \label{thm:upper}
Suppose that an even integer $d \in \N$ has $m\ge 3$ distinct 
prime factors. Then,  
\begin{equation} \label{eq:upper1}
M(d) \le 2^{m+1} + 1.
\end{equation}
Furthermore, if $d$ is square-free, then
\begin{equation} \label{eq:upper2}
M(d) \le 
\begin{cases}
13 & \textrm{if $m=3$}, \\
2^{m+1} + 7 - 4m & \textrm{if $m\ge 4$.}
\end{cases}
\end{equation}
\end{theorem}

\begin{proof}
We first define the subset of factors of $d$: 
$$
\cD(d) = \{j:\, j\mid d, \gcd(j,d/j)=1,1<j<d\}. 
$$
Since $d$ has $m$ distinct prime factors, where $m\ge 3$, we have 
$$
|\cD(d)| = \binom{m}{1} + \binom{m}{2} + \cdots + \binom{m}{m-1} = 2^m - 2. 
$$

From \eqref{eq:Pillai1},  since $1 \le x<y$ and $d=g^x(g^{y-x}+1)$, in view of $\gcd(g^x,g^{y-x}+1)=1$, 
we obtain $g^x \in \cD(d)$. By the same argument, from
 \eqref{eq:Pillai2} it follows that $g^x \in \cD(d)$. 
However, since $d$ is not of the form $2^r \cdot 3$, 
there are no positive integer $g \ge 2$ and non-negative integers $x,u,v$ for which
$$
d = g^x(g^u+1) = g^x(g^v-1). 
$$
This means that $g^x$ counted as a primitive solution $(g,x,y)$ in $N^{+}(d)$ and $g^x$ similarly counted in $N^{-}(d)$ are distinct. 
Thus, we obtain 
$$
N^{+}(d) + N^{-}(d) \le |\cD(d)| = 2^m - 2. 
$$
Therefore, applying Lemma~\ref{lem:size}, we deduce that 
$$
M(d) = 2N^{+}(d) + 2N^{-}(d) + 5 \le 2^{m+1} + 1. 
$$
This completes the proof of \eqref{eq:upper1}. 

From the above discussion, we see that there is an injective map, say $\sigma$, from the primitive integer solutions of \eqref{eq:Pillai1} or \eqref{eq:Pillai2} 
to the set $\cD(d)$ that sends $(g,x,y)$ to $g^x$. 
To prove the second part in \eqref{eq:upper2}, we need to show that there are $m$ elements in $\cD(d)$ which are not in the image of $\sigma$ when $m \ge 4$.  
Now, we assume that $d$ is square-free with the following prime factorization 
$$
d = p_1p_2 \cdots p_m, \qquad p_1=2 < p_2 < \cdots < p_m.  
$$

 We first claim that the cases $g^x = d / p_i$, $1\le i \le m-1$, cannot happen neither in \eqref{eq:Pillai1} nor in \eqref{eq:Pillai2}. 
 Indeed, fix $p_i$, where $ i<m$. If the equation \eqref{eq:Pillai1} has an integer solution with $g^x = d/p_i$, then we must have $g=d/p_i$ and $x=1$. Thus,
 by $d=g^y+g^x$ and by the choice of $p_i$, we obtain $y=1$, which contradicts to $y>x$. 
 Similarly, we can show that the equation \eqref{eq:Pillai2} has no integer solution $(g,x,y)$ for which $g^x = d/p_i$. 
 This proves the claim, and this claim actually shows that these $m-1$ elements ($d/p_i,i=1,2,\ldots,m-1$) in $\cD(d)$ are not 
 in the image of $\sigma$. Hence, we have 
 $$
M(d) \le 2(|\cD(d)| - (m-1)) +5  = 2^{m+1} +3 -2m.
$$
In particular, this implies the first part of \eqref{eq:upper2} when $m=3$. 
 
 To complete the proof, we only need to exclude $m-2$ more cases when $m\ge 4$. 
 For any $2\le i < m$, as the above,  both equations \eqref{eq:Pillai1} and \eqref{eq:Pillai2} have no integer solution with $g^x = d/(p_1p_i)$, 
 where we need to use $m \ge 4$.  
 So, this shows that these $m-2$ elements ($d/(p_1p_i),i=2,3,\ldots,m-1$) in $\cD(d)$ are not in the image of $\sigma$. 
 This in fact completes the proof. 
\end{proof}

We remark that the estimate \eqref{eq:upper2} is optimal in general. 
For example, $M(30)=13$, which achieves the first upper bound in \eqref{eq:upper2}. 
In fact, $\cM(30)$ consists of the following $13$ vectors: 
\begin{align*}
& (-15,15), (-1,29), (-29,1), (1,31), (-31,-1), (-5,25), (-25,5),\\
& (-3,27), (-27,3),  (2,32), (-32,-2), (6,36), (-36, -6). 
\end{align*}
Here, except for the five vectors in the set $\{(-15,15)\} \cup S_0$, we have eight more vectors in view of 
$$30=5^2+5=3^3+3=6^2-6=2^5-2,$$
so that $N^{+}(30)=N^{-}(30)=2$.

\subsection{Conditional upper bound}

Actually, under the \textit{ABC conjecture}, there is a uniform upper bound for $M(d)$ where $d \in \N$.  
To show this, we need some preparations. 

Recall that the $ABC$ conjecture asserts that for a given real $\varepsilon>0$ there exists a constant $K_\varepsilon$ 
depending only on $\varepsilon$ such that for any non-zero integers $A,B,C$ satisfying
$$
A+B=C
$$
and $\gcd(A,B)=1$ we have
$$
\max \{ |A|,|B|,|C|\} \le K_\varepsilon \big( \prod_{p \mid ABC}p \big)^{1+\varepsilon},
$$
where $p$ runs through all the (distinct) prime factors of $ABC$. 

We first show an unconditional result, which is an analogue of \cite[Theorem 6.2]{Bugeaud}. 

\begin{lemma}  \label{lem:ab}
Assume that $x_1,x_2,y_1,y_2$ are fixed positive integers with $x_1>x_2,y_1>y_2,x_1>y_1, \gcd(x_1,x_2)=1$ and $\gcd(y_1,y_2)=1$. 
Then, the equation 
\begin{equation}  \label{eq:ab1}
a^{x_1} + a^{x_2} = b^{y_1} + b^{y_2}
\end{equation}
has only finitely many positive integer solutions $(a,b)$. 
\end{lemma} 

\begin{proof}
Note that, since $x_1>y_1$ and $y_1>y_2 \ge 1$, we have $x_1> y_1 \ge 2$. 
If $y_1\ge 3$, then, by \cite[Theorem 1]{Peter}, the equation 
$$
a^{x_1} + a^{x_2} = b^{y_1} + b^{y_2}
$$
has only finitely many positive integer solutions $(a,b)$. 

Next, let $y_1=2$. Then, $y_2=1$, and thus the equation \eqref{eq:ab1} becomes 
\begin{equation}   \label{eq:ab2}
a^{x_1} + a^{x_2} = b^2 +b 
\end{equation}
with unknowns $a,b$. 
If $x_1=2x_2$, then, since $\gcd(x_1,x_2)=1$, we must have $x_1=2$, 
which contradicts with $x_1>y_1=2$. So, we can assume that $x_1 \ne 2x_2$.
Then, using \cite[Theorem 2]{Peter} and noticing $x_1 \ge 3$, we only need to consider the following cases: 
\begin{equation}  \label{eq:x1x2}
\textrm{$(x_1,x_2)=(3,1), (3,2), (4,1), (4,3), (6,2)$, and $(6,4)$}. 
\end{equation} 
In order to apply Lemma \ref{lem:Siegel}, we rewrite \eqref{eq:ab2} as 
\begin{equation}   \label{eq:ab3}
 4a^{x_1} + 4a^{x_2} + 1 = (2b+1)^2. 
\end{equation}
For any case of $(x_1,x_2)$ listed in \eqref{eq:x1x2}, the left-hand side of \eqref{eq:ab3} is in fact a  polynomial 
in $a$. By
computing its discriminant, one can see that it is non-zero, so the polynomial $4a^{x_1} + 4a^{x_2} + 1$ has at least three simple roots. Thus, by Lemma \ref{lem:Siegel}, 
the equation \eqref{eq:ab3} has only finitely many integer solutions $(a,b)$. This completes the proof of the lemma. 
\end{proof}

The following lemma is a direct analogue of \cite[Theorem 6.1]{Bugeaud},
where the equation $a^{x_1} - a^{x_2} = b^{y_1} - b^{y_2}$
instead of \eqref{eq:ab4} have been considered. 

\begin{lemma} \label{lem:ABC}
Under the $ABC$ conjecture, the equation 
\begin{equation}  \label{eq:ab4}
a^{x_1} + a^{x_2} = b^{y_1} + b^{y_2}
\end{equation}
has only finitely many positive integer solutions $(a,b,x_1,x_2,y_1,y_2)$ with $a>1,b>1,x_1>x_2,y_1>y_2$ and $a^{x_1} \ne b^{y_1}$. 
\end{lemma}

\begin{proof}
First, applying the same arguments  as those in Step 1 and Step 2 of the proof of \cite[Theorem 6.1]{Bugeaud}, we can prove that, 
under the $ABC$ conjecture, both $x_1$ and $y_1$ are bounded from above. 

Next, let us fix positive integers $x_1,x_2,y_1,y_2$, where $x_1>x_2, y_1>y_2$. 
If $\gcd(x_1,x_2)>1$, then in \eqref{eq:ab4} we can replace $a$ by $a^{\gcd(x_1,x_2)}$. 
So, without loss of generality, we can assume that $\gcd(x_1,x_2)=1$ and $\gcd(y_1,y_2)=1$.  
If $x_1=y_1$, then by $a^{x_1} \ne b^{y_1}$ we have $a\ne b$, say $a>b$, and so 
$$
a^{x_1} + a^{x_2} > a^{x_1} \ge (b+1)^{x_1} = (b+1)^{y_1} > b^{y_1} + b^{y_2}, 
$$
which implies that there is no such integer solution $(a,b)$. 
Thus, we can further assume that $x_1 \ne y_1$, say, $x_1 > y_1$. 
Then, by Lemma \ref{lem:ab}, the equation 
\begin{equation*} 
a^{x_1} + a^{x_2} = b^{y_1} + b^{y_2}
\end{equation*}
has only finitely many positive integer solutions $(a,b)$. 
This concludes the proof. 
\end{proof} 

The next corollary follows from Lemma~\ref{lem:ABC} and \cite[Theorem 6.1]{Bugeaud}. 

\begin{corollary} \label{cor:N+-}
Under the $ABC$ conjecture, for each sufficiently large $d$ we have $N^{+}(d) \leq 1$ and $N^{-}(d) \leq 1$. 
\end{corollary} 

\begin{proof}
By Lemma~\ref{lem:ABC}, under the $ABC$ conjecture, there are only finitely many positive integer solutions of \eqref{eq:ab4}. 
So, excluding these solutions, for large enough $d$ there will be no solutions $(a,b,x_1,x_2,y_1,y_2)$ of the equation 
           $a^{x_1} + a^{x_2} = b^{y_1} + b^{y_2} = d$
with restrictions as in Lemma~\ref{lem:ABC}. 
This yields $N^{+}(d) \le 1$ for $d$ large enough. Similar argument implies $N^{-}(d) \le 1$, by \cite[Theorem 6.1]{Bugeaud}.
\end{proof}

We are now ready to give a conditional uniform upper bound for $M(d)$. 

\begin{theorem}  \label{thm:uniform}
Under the $ABC$ conjecture, there is a positive integer $C_1$ such that for any integer $d \in \N$ we have $M(d)\le C_1$. Moreover, under the $ABC$ conjecture, we have
$M(d) \le 9$ for $d$ large enough. 
\end{theorem}

\begin{proof}
Take any $d_1$ such that for $d \ge d_1$ the two inequalities in Corollary~\ref{cor:N+-} hold. Set $C_2=\max_{1 \le d < d_1} N^{+}(d)$ and $C_3=\max_{1 \le d < d_1} N^{-}(d)$. (Evidently, we have $C_2<\infty$ and $C_3<\infty$ by Theorems \ref{thm:size2}, \ref{thm:size23} and \ref{thm:upper}.) Therefore, Lemma~\ref{lem:size} implies that $$M(d) \le 2C_2+2C_3+5.$$
This proves the first assertion of the theorem with $C_1=2C_2+2C_3+5$.
For $d \ge d_1$ we have $M(d) \leq 2+2+5=9$, by Corollary~\ref{cor:N+-} and Lemma~\ref{lem:size}, which proves the second assertion of the theorem.
\end{proof}

In Conjecture \ref{conj:Md} below we predict that the integer $C_1$ in Theorem \ref{thm:uniform} can be chosen to be 13 
according to the numerical data. Note that for $d$ large enough 
the constant $9$ of Theorem~\ref{thm:uniform} would be best possible. 
To see this, we can take $d= 3 \cdot 2^r$ with $r \ge 4$. With this choice,  by Theorem \ref{thm:size23} we have $M(d)=9$ for each such $d$. 
Also, we can take $d$ of the form
$n^2+n$, where $n \ge 2$. Then, for each such $d$ we have $N^{+}(d) \ge 1$. Indeed, this is true if $n$ is not a perfect power. If it is, say $n=g^m$, where $m \ge 2$ and $g \ge 2$ is not a perfect power, we still have $N^{+}(d) \ge 1$
in view of $d=g^{2m}+g^m$. By the same argument, the inequality $N^{-}(d) \ge 1$ holds, since $$d=n^2+n=(n+1)^2-(n+1).$$ Consequently, $M(d) \ge 9$ for each $d$ of the form $n^2+n$, $n \ge 2$.  

\subsection{Numerical data and conjectures}

In this section, we want to design an algorithm for computing $M(d),d\in \N$, and perform the corresponding computations. 

From Theorem~\ref{thm:size2} (ii), we only need to compute $M(d)$ for positive even integers $d$. 
Based on Lemma~\ref{lem:size}, we design Algorithm~\ref{alg:Md} for this purpose. 
As one can see, the algorithm is very simple, and essentially it is also an algorithm to solve the equations \eqref{eq:Pillai1} and \eqref{eq:Pillai2}.  
Here, we use PARI/GP \cite{Pari} to implement this algorithm and make the corresponding computations.

\begin{algorithm}
\caption{Computing $M(d)$}
\label{alg:Md}
\begin{algorithmic}[1]
\Require positive even integer $d \ge 4$ (input).
\Ensure $M(d)$ (output).
\State Compute the prime factorization of $d$, say, $d=p_1^{r_1}p_2^{r_2} \cdots p_m^{r_m}$. 
\State Set $A,B$ to be two zero vectors of size $2^m$. 
\State Execute the subsequent three steps by running through all the factors $a$ of $d$ with $\gcd(a,d/a)=1$. 
\State Given such a factor $a$ of $d$, say $a=p_1^{r_1}\cdots p_j^{r_j}$, compute $r=\gcd(r_1,\ldots,r_j)$ and $g=p_1^{r_1/r}\cdots p_j^{r_j/r}$. 
\State Divide $d-a$ repeatedly by $g$ until the quotient is not greater than 1. Then, if the quotient is equal to 1, store $a$ in the vector $A$. 
\State Divide $d+a$ repeatedly by $g$ until the quotient is not greater than 1. Then, if the quotient is equal to 1, store $a$ in the vector $B$. 
\State Count the number of distinct non-zero entries in $A$, say $N_1$, and count the number of distinct non-zero entries in $B$, say $N_2$. 
Return $M(d) = 2(N_1 + N_2) +5$. 
\end{algorithmic}
\end{algorithm}

When using Algorithm~\ref{alg:Md} to compute $M(d)$ for a large range of $d$, to speed up the computation and save the memory we can set $A,B$ to be two zero vectors of size 2 in Step 2 of Algorithm~\ref{alg:Md}, and then let the algorithm return the value of $d$ if the size 2 is not big enough. 
Besides, in Step 3 of Algorithm~\ref{alg:Md} we use the binary representations of integers between 0 and $2^m-1$ to run over all such $2^m$ factors of $d$. For example, the factor corresponding to the binary number $0\ldots 011$ is $p_1^{r_1}p_2^{r_2}$.

In Table \ref{ta:Md}, the first row shows all the possible values of $M(d)$ for positive even integer $d\le 10^{10}$.
 The second row gives the number of such integers $d\le 10^3$ whose $M(d)$ correspond to the values in the first 
 row. Other rows
 have similar meaning.

\begin{table}
\centering
\caption{Statistics of $M(d)$ for positive even integers $d$}
\label{ta:Md}
\begin{tabular}{|c|c|c|c|c|c|}
\hline
$M(d)$ & 5 & 7 & 9 & 11 & 13 \\ \hline

$d\le 10^3$ & 380  & 79  & 33  & 7 & 1 \\ \hline 

$d\le 10^4$ & 4653 & 233  & 103  & 10 & 1 \\ \hline 

$d\le 10^5$ & 49177  & 488  &  323 & 11 & 1 \\ \hline

$d\le 10^6$ & 498015  & 963  & 1010  & 11 & 1 \\ \hline 

$d\le 10^7$ & 4994967  & 1846  & 3175  & 11 & 1 \\ \hline

$d\le 10^8$ & 49986562 & 3410  & 10015  & 12 & 1 \\ \hline 

$d\le 10^9$ & 499961918 &  6427 & 31642  & 12 & 1 \\ \hline

$d\le 10^{10}$ & 4999887540  & 12425  & 100022  & 12 &  1 \\ \hline
\end{tabular}
\end{table}

In particular, we have $M(30)=13$, and $M(d)=11$ if $d$ is one of the following twelve integers: 
$$
6, 12, 24, 132, 210, 240, 252, 6480, 8190, 9702, 78120, 24299970.
$$
In fact, these thirteen integers are of the form $n^2+n$ except for $d=24$ and $d=252$. For example, 
$24299970 = 4929^2+ 4929$. 
Moreover, we used Algorithm \ref{alg:Md} to test all the integers $d=n^2+n$, where $4930 \le n \le 10^8$, and found no examples with $M(d) > 9$.

Furthermore, from Table~\ref{ta:Md} and Theorem~\ref{thm:size2} we see that for any positive integer $d\le 10^{10}$ we have 
$$
M(d) \le 13.
$$

From Table~\ref{ta:Md}, one can also observe the following interesting pheno\-me\-non. 
Corresponding to the values $5, 7, 9$, the quotients of the numbers of such integers $d$ in two nearby rows 
are very close to $10, 2, 3$, respectively.  

Based on our computations, we pose two conjectures on the equations \eqref{eq:Pillai1} and \eqref{eq:Pillai2} as follows, 
which are of independent interest.   

\begin{conjecture} \label{conj:Pillai1}
For any given integer $d\ge 1$, we have $N^{+}(d) \le 2$.
\end{conjecture}

\begin{conjecture} \label{conj:Pillai2}
For any given integer $d\ge 1$, we have $N^{-}(d) \le 2$.
\end{conjecture}

\begin{table}
\centering
\caption{The values of $d\le 10^{10}$ with $N^{+}(d) = 2$}
\label{ta:N+}
\begin{tabular}{|c|c|}
\hline
$d$ & \textrm{Primitive integer solutions $(g,x,y)$ of \eqref{eq:Pillai1}} \\ \hline

$12$ & (2,2,3), (3,1,2)  \\ \hline 

$30$ & (3,1,3), (5,1,2)  \\ \hline 

$36$ & (2,2,5), (3,2,3)  \\ \hline 

$130$ & (2,1,7), (5,1,3)  \\ \hline

$132$ & (2,2,7), (11,1,2)  \\ \hline

$252$ & (3,2,5), (6,2,3)  \\ \hline

$9702$ & (21,2,3), (98,1,2)  \\ \hline

$65600$ & (2,6,16), (40,2,3)  \\ \hline

\end{tabular}
\end{table}

From our computations, it follows that Conjectures \ref{conj:Pillai1} and \ref{conj:Pillai2} are true for 
all positive integers $d\le 10^{10}$. 
Moreover, it is likely that either $N^{+}(d) = 2$ or $N^{-}(d) = 2$ are very rare events. We collect the values of positive integers $d\le 10^{10}$ for which either $N^{+}(d) = 2$ or $N^{-}(d) = 2$, and the corresponding primitive integer solutions of  the equations \eqref{eq:Pillai1} and \eqref{eq:Pillai2} in Tables \ref{ta:N+} and \ref{ta:N-}, respectively. 
In particular, one can see that 30 is the unique positive integer in the range $[1,10^{10}]$ with $N^{+}(30) = 2$ and $N^{-}(30) = 2$. 
We emphasize that, by Corollary \ref{cor:N+-}, under the $ABC$ conjecture the inequalities $N^{+}(d) \leq 1$ and $N^{-}(d) \leq 1$ hold for each sufficiently large $d$.  The last example in Table~\ref{ta:N-} corresponds
to the solution
$(x,y)=(30,9859)$ on the hyperelliptic curve
$$y^2=4x^5-4x+1.$$ Inserting $y=2 \cdot 4930 -1$ and $x=30$ we get
$4930^2-4930=30^5-30$.

\begin{table}
\centering
\caption{The values of $d\le 10^{10}$ with $N^{-}(d) = 2$}
\label{ta:N-}
\begin{tabular}{|c|c|}
\hline
$d$ & \textrm{Primitive integer solutions $(g,x,y)$ of \eqref{eq:Pillai2}} \\ \hline

$6$ & (2,1,3), (3,1,2)  \\ \hline 

$24$ & (2,3,5), (3,1,3)  \\ \hline 

$30$ & (2,1,5), (6,1,2)  \\ \hline 

$120$ & (2,3,7), (5,1,3)  \\ \hline 

$210$ & (6,1,3), (15,1,2)  \\ \hline 

$240$ & (2,4,8), (3,1,5)  \\ \hline 

$2184$ & (3,1,7), (13,1,3)  \\ \hline 

$6480$ & (3,4,8), (6,4,5)  \\ \hline 

$8190$ & (2,1,13), (91,1,2)  \\ \hline 

$78120$ & (5,1,7), (280,1,2)  \\ \hline 

$24299970$ & (30,1,5), (4930,1,2)  \\ \hline 

\end{tabular}
\end{table}

From the proof of Theorem \ref{thm:uniform} we know that, under the $ABC$ conjecture, 
there exists a positive integer $C_4=\max\{C_2,C_3\}$, which is independent of $d$, such that each of the equations in Conjectures \ref{conj:Pillai1} and \ref{conj:Pillai2} has at most $C_4$ primitive integer solutions. 

Under Conjectures \ref{conj:Pillai1} and \ref{conj:Pillai2} and in view of \eqref{countt}, for any integer $d  \in \N$ we have  
$$
M(d) \le 13,
$$
which is also compatible with our numerical data. 
So, in conclusion we suggest the following conjecture. 

\begin{conjecture}  \label{conj:Md}
For any $d \in \N$ we have 
$M(d) \le 13.$
Moreover, $M(d)=13$ if and only if $d=30$.
\end{conjecture}

In fact, the second part of Conjecture \ref{conj:Md} asserts that 30 is the unique positive integer $d$ satisfying $N^{+}(d) = N^{-}(d) = 2$.

\section*{Acknowledgement} 

The research of M.~S. was supported by the Macquarie University Research Fellowship. 
The authors thank Dr. Alina Ostafe for corresponding the relevant reference \cite{BMZ}. 
They are also grateful to the referee who read the paper very carefully and pointed out several errors and inaccuracies.

\end{document}